\documentclass[a4paper,10pt]{article} \usepackage[latin1]{inputenc}
\usepackage{latexsym,amsmath, amscd, amsthm, amssymb, amsfonts} \usepackage{xy}
\usepackage[dvips]{graphicx}
\theoremstyle{plain}
\usepackage{multirow,array} 
\newtheorem{theorem}{Theorem}[section]
\newtheorem{remark}[theorem]{Remark}
\newtheorem{lemma}[theorem]{Lemma}
\newtheorem{proposition}[theorem]{Proposition}

\theoremstyle{definition}
\newtheorem{example}[theorem]{Example}
\DeclareMathOperator{\ini}{in}

\DeclareMathOperator{\nf}{Nf}
\DeclareMathOperator{\Undone}{Undone}
\DeclareMathOperator{\Done}{Done}

\title{Vector space bases associated
 to vanishing ideals of points}
\author{Samuel Lundqvist}
\date{}
\begin{document}
\maketitle

\begin{abstract}
In this paper we discuss four different constructions of vector space bases associated to vanishing ideals of points. We show
how to compute normal forms with respect to these bases and give new complexity bounds. As an application, we drastically improve the computational algebra approach to the reverse engineering of gene regulatory networks.
\end{abstract}

\section{Introduction}
Let $\Bbbk[x_1, \ldots, x_n]$ be the polynomial ring in $n$ variables over a field $\Bbbk$. 
The vanishing ideal
with respect to a set of points $\{p_1, \ldots, p_m\}$ in $\Bbbk^n$ is defined as the set of elements in $\Bbbk[x_1, \ldots, x_n]$ 
that are zero on all of the $p_i$'s. 

The main tool that is used to compute vanishing ideals of points is the Buchberger-M\"oller algorithm, described in 
\cite{BM}. The Buchberger-M\"oller algorithm returns a Gr\"{o}bner basis for the ideal vanishing on the set $\{p_1,\ldots, p_m\}$. 
A complementary result of the algorithm is a vector space basis for the quotient ring $\Bbbk[x_1, \ldots,
x_n]/I$. However, in many applications it turns out that it is the vector
space basis, rather than the Gr\"{o}bner basis of the ideal, which is of
interest. For instance, it may be preferable to compute normal forms using vector space methods instead of
Gr\"{o}bner basis techniques.

A new bound for the arithmetic complexity of the Buchberger-M\"oller algorithm is given in \cite{lundqvist}, and is equal to
$O(nm^2 + \min(m,n)m^3)$. We will discuss four constructions of
vector space bases, all of which perform better than the Buchberger-M\"oller algorithm. 
An application of the constructions will be that we can improve the method of the reverse engineering 
of gene regulatory networks given in \cite{LaubStigler04}.

A key result for the effectiveness of our methods is a fast
combinatorial algorithm which gives useful structure information about
the relation between the points. The algorithm uses less than $nm+m^2$ arithmetic
comparisons over $\Bbbk$.

As a direct consequence of the combinatorial algorithm, we obtain with
the same complexity a family of separators, that is, a family $\{f_1,
\ldots, f_m\}$ of polynomials such that $f_i(p_i) = 1$
and $f_i(p_j) = 0$ if $i \neq j$. It is easy to see that the
separators form a $\Bbbk$-basis for the quotient ring $\Bbbk[x_1, \ldots, x_n]/I$.
This will be the first construction of vector space bases.

The second construction is a $\Bbbk$-basis formed by the residues of
$1,f, \ldots, f^{m-1}$, where $f$ is a linear form. 
If
$\Bbbk$ is equipped with a total order, this construction 
uses $O(nm + \min(m,n)m^2 \log(m))$ arithmetic operations Also, we obtain an algebra isomorphism $S/I \cong \Bbbk[x]/J$, where
$J$ is a principal ideal.

The two remaining constructions give monomial $\Bbbk$-bases. The
third construction we discuss is a method that was introduced in \cite{CerliencoMureddu} and
improved in \cite{Felszeghy}. It produces
the set of monomials outside the initial ideal of $I$ with respect to the lexicographical ordering, using 
only combinatorial methods. We reanalyze the number of arithmetic operations needed 
in the method presented in \cite{Felszeghy} and we show that it uses only 
$O(nm+m\min(m,nr))$ arithmetic comparisons (the variable $r$ is less than $m$ and will be defined later).

The fourth construction gives a $\Bbbk$-basis which is the complement of
the initial ideal with respect to a class of admissible monomial
orders in a total of $O(nm)$ arithmetic comparisons and additionally $O(\min(m,n)m^3)$ 
arithmetic operations. 

In Section \ref{sec:bio}, we end up by giving the biological implications 
of our constructions.

\section{Notation and preliminaries} \label{sec:notation} Throughout
the paper, let $S=\Bbbk[x_1, \ldots, x_n]$ denote the
polynomial ring in $n$ variables over a field $\Bbbk$ and let $I$ denote an ideal in $S$. Whenever $f \in S$, $|f|$ will denote the
degree of $f$ and $[f]$ will denote the equivalence class in
$S/I$ containing $f$.

Let $B$ be any subset of $S$ such that $[B] = \{[b], b \in B\}$ is a
vector space basis for $S/I$. If $s$ is an element in $S$, its
residue can be uniquely expressed as a linear combination of the
elements in $[B]$, say $[s] = \sum c_i [b_i]$. The $S$-element $\sum
c_i b_i$ is then called the normal form of $s$ with respect to $B$ which
we write as $\nf(s,B) = \sum c_i b_i$. By abuse of notation we say that
$B$ (rather than $[B]$) is a basis for $S/I$.

Let $p$ be a point in $\Bbbk^n$ and $f$ an element of $S$. We denote
by $f(p)$ the evaluation of $f$ at $p$. When $P= \{p_1, \ldots,
p_m\}$ is a set of points, we write $f(P) = (f(p_1), \ldots, f(p_m))$. If $F =
\{f_1, \ldots, f_s\}$ is a set of elements in $S$, then $F(P)$ is
defined to be the $(s \times m)$-matrix whose $i$'th row is $f_i(P)$.

The vanishing ideal $I(P)$ is the ideal consisting of all elements in
$S$ which vanish on all of the points in $P$. 
Given two elements $f_1$ and $f_2$ in $S$ such that $[f_1] = [f_2]$ in $S/I(P)$, we have $f_1(p) =
f_2(p)$ for all $p \in P$. An important property of a set $[B]$ of $m$
elements being a $\Bbbk$-basis for $S/I(P)$ is that $B(P)$ has rank $m$.

A family of separators for a set of distinct points is a set $\{f_1,
\ldots, f_m\}$ of polynomials such that $f_i(p_i) = 1$ and $f_i(p_j) =
0$ whenever $i \neq j$. The residues of a family of separators will
always form a $\Bbbk$-basis for $S/I$. Such a $\Bbbk$-basis will have a nice 
property for computing normal forms and we have the formula
$\nf(f,Sep) = f(p_1) \cdot f_1 + \cdots + f(p_m) \cdot f_m$.

An admissible monomial order is a total order on the monomials in $S$ which also is a well ordering. 
Let $\prec$ be an admissible monomial order. The initial ideal of
$I$, denoted by $\ini(I)$, is the monomial ideal consisting of all
leading monomials of $I$ with respect to $\prec$. One of the
characterizations of a set $G$ being a Gr\"{o}bner basis of an ideal $I$
with respect to an admissible monomial order $\prec$ is that $G \subseteq I$ and
that the leading terms of $G$ generate $\ini(I)$. An old theorem by
Macaulay states that the residues of the monomials outside $\ini(I)$
form a $\Bbbk$-basis for the quotient ring $S/I$. The set of monomials
outside $\ini(I)$ will be called the \emph{standard monomials} (with respect to $\prec$) throughout the paper.

We will measure the performance of the algorithms presented by means
of the number of arithmetic comparisons and the number of arithmetic
operations (addition and multiplication). We will assume that the
cost of an arithmetic comparison is bounded by the cost of an arithmetic
operation. 

Some of the algorithms we present use comparisons and incrementions on the set $\{1, \ldots, m\}$. We call these 
operations \emph{elementary integer operations}. 
The elementary integer operations will in general be neglectable and we will, except for some cases, omit them in the complexity analysis. 

In the sequel, the word "operations" will mean arithmetic operations and
the word "comparisons" will mean arithmetic comparisons, if not stated otherwise. By "bound" we will
always mean an upper bound. 

We do not deal with the growth of coefficients in the
operations, but refer the reader to \cite{FGLM}. In
\cite{Abbottetal}, the techniques in the case when $\Bbbk = \mathbb{Q}$
are discussed, using the Chinese remainder theorem.

\section{Normal form computations for zero dimensional ideals}
\label{sec:nf}

The most frequent method for computing normal forms with respect to an ideal $I$ uses the Noetherian
property of the reduction process with respect to a Gr\"{o}bner basis of
the ideal $I$. However, if $I$ is a vanishing ideal, 
linear algebra techniques are to prefer. Indeed, the reduction process
with respect to a Gr\"{o}bner basis can have exponential runtime, while
the linear algebra techniques have low polynomial runtime. On the other hand, the
linear algebra techniques do not seem to be widely spread and we will 
describe them here. 

\begin{lemma} \label{lemma:sepfrombasis}

 Let $B = \{e_1, \ldots, e_m\}$ and suppose that $[B]$ is a basis for
 $S/I(P)$. Let
$$(f_1, \ldots, f_m)^t = B(P)^{-1} (e_1, \ldots, e_m)^t.$$
Then $\{f_1, \ldots, f_m\}$ is a family of separators.
\end{lemma}
\begin{proof}
$$B(P)^{-1} (e_1, \ldots, e_m)^t(p_i) = B(P)^{-1} (e_1(p_i), \ldots, e_m(p_i))^t$$ 
$$= 
B(P)^{-1} B(P) (\underbrace{0,\ldots,0}_{ i-1 \text{ times}},1,0,\ldots,0)^t = (\underbrace{0,\ldots,0}_{ i-1 \text{ times}},1,0,\ldots,0)^t.$$
\end{proof}

\begin{lemma} \label{lemma:nfwrtbasis}
Suppose that $[B]$ is a basis for
 $S/I(P)$. Then
$$Nf(f,B) = ([e_1], \ldots, [e_m]) (B(P)^{-1})^t (f(p_1), \ldots, f(p_m))^t.$$
\end{lemma}
\begin{proof}
 We have $\nf(f,Sep) = (f_1, \ldots, f_m) (f(p_1),\ldots,
 f(p_m))^t$ and by Lemma \ref{lemma:sepfrombasis},
$(f_1, \ldots, f_m) = (e_1, \ldots, e_m) (B^{-1})^t.$
\end{proof}
Since evaluation of a monomial of
degree $d$ at a point $p$ is done using $d$ multiplications, the
complexity of evaluating $f$ at the $m$ points uses $O(|f|sm)$
operations, where $s$ is the number of monomials in $f$.
Multiplication by $B(P)$ requires an additional number of
$O(m^2)$ operations, so we have proven the following proposition.
\begin{proposition} \label{prop:normalformPoints} Let $B = \{e_1,
 \ldots, e_m\}$ and suppose that $[B]$ is a basis for $S/I$, where
 $I$ is the vanishing ideal with respect to the points $p_1, \ldots,
 p_m$. Suppose that $B(P)$ and the inverse of $B(P)$ have been
 computed. Then we have a normal form algorithm with respect to the
 separators which runs in $O(|f|sm)$ operations, where $f$
 is a polynomial with $s$ monomials. To compute the normal form with
 respect to the basis $B$, we need to perform $O(|f|sm + m^2)$ 
 operations.
\end{proposition}
We can also use the theory of multiplication matrices, described in \cite{corless} for instance, to compute normal forms
of vanishing ideals of points with the same complexity as above. The theory of multiplication
matrices also allows to extend normal form algorithms to general rings $S/I$, where
$\dim_k(S/I) < \infty$, see for instance the Mathphi-algorithm in \cite{FGLM}.

\section{Combinatorial preprocessing of the points} \label{sec:witness}
All of the constructions in Section \ref{sec:constructions} rely on a combinatorial preprocessing of the points that we will described here. We will give the notations and the results here, but refer the reader to the Appendix for the algorithmic study.


Let $\Omega$ be a set equipped with an equivalence relation, denoted by $=$. The
 equivalence relation on $\Omega$ is extended to $n$-tuples of
 elements in $\Omega$ by $a = (a_1, \ldots, a_n) = (b_1, \ldots, b_n) = b$ if
 $a_i = b_i$, for all $i$. The \emph{witness} of two different n-tuples $a$ and $b$ is the least $i$ such that
$a_i \neq b_i$. When $a=b$, the witness is zero.
 Let $\pi_i$ be the projection map from $\Omega^n$ to $\Omega^i$ given by 
 $(a_1, \ldots, a_n) \mapsto (a_1, \ldots, a_i)$. 
 Let $v_1, \ldots, v_m$ be $n$-tuples of elements in $\Omega$ and 
 let $\Sigma_i$ be the set of equivalence classes of 
 $\pi_i(v_1), \ldots, \pi_i(v_m)$ for $i=1, \ldots, n$. 
 To simplify the notation, we will represent an equivalence class of tuples as an
 index set, that is, as a subset of $\{1, \ldots, m\}$ instead of as a
 subset of $\{v_1, \ldots, v_m\}$. Using this notation, we define 
$\Sigma_0 = \{\{1, \ldots, m\}\}$. 

Let $\overline{m}$ be the number of distinct elements in the set
 $\{v_1, \ldots, v_m\}$. Notice that $\overline{m} = |\Sigma_n|$. 
 Let $W$ be the \emph{witness list} - the set of all $i$, $i \in \{ 1, \ldots, n \}$, such that
 $\Sigma_{i-1} \neq \Sigma_{i}$. Notice that $W$ is the set of witnesses. 
 Finally, let $C$ be the \emph{witness matrix} - an upper triangular matrix with elements in 
$W \cup \{0\}$ such that, for $i<j$, the number $c_{ij}$ is the witness of $v_i$ and $v_j$. 

\begin{example} \label{example1}

 In $\Omega = \mathbb{Z}$, let
 $v_1 = (1,2,0,1,1,0,3,5),$
 $v_2 = (1,0,1,1,2,0,3,5),$
 $v_3 = (1,2,0,3,3,1,2,0),$
 $v_4 = (0,0,2,0,4,0,2,0),$
 $v_5 = (0,0,2,1,5,0,2,0)$ and $v_6 =$ $(2,1,3,1,6,0,2,0).$ We will write the
 vectors as columns in the left hand side of the table below. In the right hand side we write the equivalence classes. 
 \begin{displaymath}
\left( \begin{array}{cccccc|lc}
 1 & 1 & 1 & 0 & 0 & 2 & \Sigma_1 = & \{\{1,2,3 \}, \{4,5 \}, \{6 \} \} \\
 2 & 0 & 2 & 0 & 0 & 1 & \Sigma_2 = & \{\{1,3\}, \{2\}, \{4,5\}, \{6\} \} \\
 0 & 1 & 0 & 2 & 2 & 3 & \Sigma_3 = &\{\{1,3\}, \{2\}, \{4,5\}, \{6\} \} \\
 1 & 1 & 3 & 0 & 1 & 1 & \Sigma_4 = &\{\{1 \}, \{3 \}, \{2\}, \{4\}, \{5\}, \{6\} \}\\
 1 & 2 & 3 & 4 & 5 & 6 & \Sigma_5 = &\{\{1 \}, \{3 \}, \{2\}, \{4\}, \{5\}, \{6\} \}\\
 0 & 0 & 1 & 0 & 0 & 0 & \Sigma_6 = &\{\{1 \}, \{3 \}, \{2\}, \{4\}, \{5\}, \{6\} \}\\
 3 & 3 & 2 & 2 & 2 & 2 & \Sigma_7 = &\{\{1 \}, \{3 \}, \{2\}, \{4\}, \{5\}, \{6\} \}\\
 5 & 5 & 0 & 0 & 0 & 0 & \Sigma_8 = &\{\{1 \}, \{3 \}, \{2\}, \{4\}, \{5\}, \{6\} \}
\end{array} \right).
 \end{displaymath}
For instance, $\{1,3\} \in \Sigma_2$ shows that $p_1$ and $p_3$ agree on the first two coordinates. 
We have $\overline{m} = 6$, $W = \{1,2,4\}$ and 
$$C=
\begin{pmatrix}
 0 & 2 & 4 & 1 & 1 & 1 \\
 0 & 0 & 2 & 1 & 1 & 1 \\
 0 & 0 & 0 & 1 & 1 & 1 \\
 0 & 0 & 0 & 0 & 4 & 1 \\
 0 & 0 & 0 & 0 & 0 & 1 \\
 0 & 0 & 0 & 0 & 0 & 0
\end{pmatrix}.
$$
\end{example}

From the $\Sigma_i$'s, we can obtain a tree
representation of the vectors. The vertices are labelled by the elements in the $\Sigma_i$'s and 
there is an edge from a vertex labelled by $\Sigma_{ik} \in \Sigma_i$ to a vertex labelled 
by $\Sigma_{i+1,h} \in \Sigma_{i+1}$ exactly when 
$\Sigma_{i+1,h} \subseteq \Sigma_{i k}$. Such an edge is labelled
by $v_{i+1, j}$, for some $j \in \Sigma_{i+1,h}$ (recall that $v_{i+1,j} = v_{i+1,j'}$ for all $j, j' \in \Sigma_{i+1,h}$). 
In this way, all vectors and paths from the root to the
leaves are in a natural one-to-one correspondence. 
The maximal number of edges from a vertex in the tree is denoted by $r$. 
In Example \ref{example1}, $r=3$ since $\Sigma_0$ has three children. 


\begin{figure}[ht!] \label{fig:trie} \centering
 \includegraphics{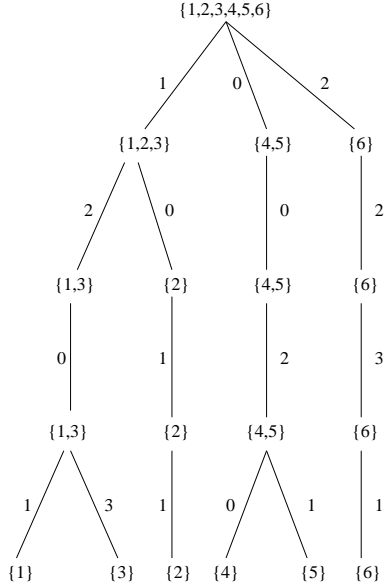}
 \caption{The point trie of the first four coordinates of Example \ref{example1}.}
\end{figure}

Figure 1 
shows how the tree representation of the first four coordinates of the
points from Example \ref{example1} look like. The authors in \cite{Felszeghy} call such a 
tree representation a \emph{trie}. 
The associated trie to a set of points is a key construction in \cite{Felszeghy, Gao}. Since we will 
deal with two different types of tries in this paper, we call the trie that is associated to a set of points 
the \emph{point trie}. 

\begin{theorem} \label{thm:cmpalg}
The $\Sigma_i$'s and/or the associated point trie can be computed using at most 
$nm + m\min(m,rn)$ comparisons. 
\end{theorem}
\begin{proof}
We refer the reader to the Appendix.
\end{proof}

\section{Vector space constructions} \label{sec:constructions}
In this section we present the four different constructions. 

\subsection{Construction 1 - A separator basis} \label{sec:conssep}
Let $I$ be the vanishing ideal of $m$ distinct points $p_1, \ldots, p_m$ and let
$C$ be the witness matrix with respect to $p_1, \ldots, p_m$. Let
\begin{equation} \label{expr:sep} 
Q_i = \prod_{j\ne i}
 \frac{x_{c_{ij}}-p_{j c_{ij}}}{p_{i c_{ij}}-p_{j c_{ij}}}.
\end{equation}
It is easily checked that $Q_i(p_i) = 1$ and $Q_i(p_j) = 0$ if $i \neq
j$. In this way we get closed expressions of a $\Bbbk$-basis for $S/I$ as a
direct consequence of Theorem \ref{thm:cmpalg}. We remark that
(\ref{expr:sep}) is a standard construction of separators. For instance, it is used in \cite{LaubStigler04}, but there, the matrix $C$
is computed in a naive way which uses $O(nm^2)$ comparisons.

\begin{theorem} \label{thm:kbasisfromwitnesses} Let $I$ be the vanishing ideal with respect
to $m$ distinct points. We can compute a set
 of separators and hence a $\Bbbk$-basis for $S/I$ in $O(nm)$ comparisons and
$O(m^2)$ operations over $\Bbbk$. We have also a normal form
 algorithm with respect to this basis. To compute the normal form of an element $f$ with $s$ monomials, 
 this algorithm runs in $O(s|f|m)$ operations over $\Bbbk$.
\end{theorem}

\begin{proof}

 By definition, $Q_1, \ldots, Q_m$ is a set of separators for the
 points and hence a $\Bbbk$-basis for $S/I$. To get the 
 $Q_i$'s is a nice form, we need to evaluate the denominators, for which
 we need $O(m^2)$ operations.
 
 To compute the witness matrix is $O(nm + \min(m,rn))$ comparisons by 
 Theorem \ref{thm:cmpalg}. Since we assume that a comparison is cheaper than an operation, we conclude that 
 we need $O(nm)$ comparisons and $O(m^2)$ operations to compute the $Q_i$'s.

The existence and effectiveness of the normal form algorithm follows from Proposition \ref{prop:normalformPoints}.
\end{proof}

\begin{example}[Example \ref{example1} continued] \label{example2}
 Consider the points $p_1, \ldots, p_6$ as elements in $\mathbb{Z}_{43}$. Using (\ref{expr:sep}), we  determine 
the separator $Q_1$ of $p_1, \ldots, p_6$ as
$$
Q_1 = \frac{x_2}{2-0} \cdot \frac{x_4 - 2}{1-2} \cdot \frac{x_1 -
 0}{1-0} \cdot \frac{x_1-0}{1-0} \cdot \frac{x_1 - 2}{1-2} =
22x_1^2(x_1-2) x_2 (x_4 -2)$$ and, similarly $$
 Q_2 = 32 x_1^2 (x_1-2) (x_2-2)^2, Q_3 = 32 x_1^2(x_1-2)x_2(x_4-1),$$ 
$$Q_4 = 22(x_1 -1)^3 (x_1-2) (x_4-1), Q_5 = 22(x_1 -1)^3 (x_1-2) x_4,$$
$$Q_6 = 11(x_1-1)^3 x_1^2.$$
Thus, $[\{Q_1, \ldots, Q_6\}]$ is a vector space basis for $\mathbb{Z}_{43}[x_1, \ldots, x_8]/I(P)$.
Let $f = x_1 x_2 x_4 + x_4 x_5 x_6x_7$. To
compute the normal form of $f$, we
evaluate the expression on the six points to get $f(P) = (2,0,24,0,0,2)$. Thus
$\nf(f, \{Q_1, \ldots, Q_6\}) = 2Q_1 + 24Q_3 + 2Q_6.$

\end{example}

\begin{remark}
If $Q_i = f_1^{\alpha_1} \cdots f_n^{\alpha_n}$ is a separator, let 
$\overline{Q_i} =  f_1 \cdots f_n$. Then $Q_i(p_j) \neq 0$ only when $i = j$, so
$\overline{Q_i}/\overline{Q_i}(p_i)$ is also a separator. This observation could be used to
compute separators of low degrees. 
\end{remark}

\subsection{Construction 2 - An isomorphism} \label{sec:consiso}

\begin{lemma} \label{lemma:diffpoints} Let $\Omega$ be a set and let
 $v_1, \ldots, v_m$ be $n$-tuples with elements in $\Omega$. Let $\overline{m}$
 denote the number of distinct tuples and, without loss of
 generality, suppose that $v_{1}, \ldots, v_{\overline{m}}$
 are distinct. Let $W = \{i_1, \ldots, i_{\overline{n}} \}$, be the 
 witness list with respect to 
 $v_1, \ldots, v_m$. 
 Let $\pi$ be the projection from
 $\Omega^n$ to $\Omega^{\overline{n}}$, defined by $\pi((a_1, \ldots,
 a_n)) = (a_{i_1}, \ldots, a_{i_{\overline{n}}})$. Let
 $w_i=\pi(v_{i})$. Then $w_1, \ldots, w_{\overline{m}}$ are distinct.
 
\end{lemma}
\begin{proof}
 Suppose that $j\neq k$ and $j,k \leq \overline{m}$. Since $v_j$ and $v_k$
 are distinct, there is a witness $i_h \in W$ such that $v_{ji_h}
 \neq v_{ki_h}$. By definition, this means that $w_j$ and
 $w_k$ differ at position $h$.
\end{proof}

\begin{lemma} \label{lemma:iso}
Let $\{p_1, \ldots, p_m\}$ be a set of distinct points in $\Bbbk^n$. Let $I$ be the vanishing ideal with respect to these points.
Let $\overline{n}$ be any positive integer and 
let $\pi$ be an algebraic map from $\Bbbk^n$ to $\Bbbk^{\overline{n}}$ such that $\pi(p_1), \ldots, \pi(p_m)$ are distinct.
Let $T = \Bbbk[y_{i_1}, \ldots y_{i_{\overline{n}}}]$ and let $J$ be the vanishing ideal with respect to $\pi(p_1), \ldots, \pi(p_m).$
Then $S/I \cong T/J$ are isomorphic as algebras.
\end{lemma}

\begin{proof}
 Let $\pi^*$ be the corresponding monomorphism from $T$ to $S$ defined by $\pi^*(f)(p) = f(\pi(p))$ for $f \in T$. Notice that $f \in J$ if and only if
$f(\pi(p_i)) = 0$ for all $i$, which is equivalent
to $\pi^*(f)(p_i) = 0$ for all $i$, which in turn holds if and only if $\pi^*(f) \in I$. This allows us to extend $\pi^*$ to a monomorphism from $T/J$ to $S/I$. Since
$\pi(p_1), \ldots, \pi(p_m)$ are all distinct, we have $\dim_{\Bbbk} (T/J) = \dim_{\Bbbk} (S/I)$ and, thus, the extension of $\pi^*$ is an isomorphism of algebras.
\end{proof}

It follows easily that if $B$ is any set such that $[B]$ is a basis for $T/J$, then 
$[\pi^*(B)]$ is a basis for $S/I$.
Let $\{p_1, \ldots, p_m\}$ be a set of distinct points and suppose that we write the points
with respect to the coordinates $(\overline{x_1},
x_2, \ldots, x_n)$, where $\overline{x_1}= g_1x_1 + \cdots + g_nx_n$
and the $g_i$'s are generic. 
Define $\pi: (a_1, \ldots, a_n)
\mapsto a_1$ with respect to these coordinates. Then the points
$\pi(p_1), \ldots, \pi(p_n)$ will be distinct. In fact, we can
replace the generic coefficients with elements in $\Bbbk$, provided
that $\Bbbk$ is large enough. We give two constructions based on this
observation.

\begin{proposition} \label{theorem:nonlinearkbasisiso} 
Let $\{p_1,\ldots, p_m\}$ be a set of distinct points in $\Bbbk^n$ and suppose that
 $\Bbbk$ contains at least $m$ elements. Let $I$ be the vanishing
 ideal with respect to these points. Then there is an algebraic map
 $\pi$ from $\Bbbk^n$ to $\Bbbk$ such that $\pi(p_1), \ldots,
 \pi(p_m)$ are distinct and $S/I \cong \Bbbk[x]/J$, where $J$ vanishes 
	on $\pi(p_1), \ldots, \pi(p_m)$.
\end{proposition}

\begin{proof}
 Let $c_1, \ldots, c_m$ be distinct elements in $\Bbbk$. Let $\{Q_1,
 \ldots, Q_m\}$ be a family of separators for the points and let further
$$\pi = \sum_i c_i Q_i.$$
Then $\pi(p_i) = c_i$. Hence, $S/I \cong \Bbbk[x]/J$ by Lemma
\ref{lemma:iso}. 
Now $[1], \ldots, [x^{m-1}]$ forms a $\Bbbk$-basis for $\Bbbk[x]/J$ and
$\pi^{*}(x) = \sum_i c_i Q_i$. It follows that 
$$[1], [\sum_i c_i Q_i], \ldots, [\sum_i c_i Q_i]^{m-1}$$ 
becomes a $\Bbbk$-basis for $S/I$.
\end{proof}

If we assume that $\Bbbk$ contains at least $\binom{m}{2}+1$ elements,
then the map $\pi$ from $\Bbbk^n$ to $\Bbbk$ can be chosen to be a
projection, and the $\Bbbk$-basis will be of the form $[1], [x],
\ldots, [x^{m-1}]$ with $x$ linear.

To settle this, we need to introduce some notation. Consider the point set $\{p_1, \ldots, p_m\}$ in $\Bbbk^n$. Recall that $p_{ik} = p_{jk}$ for all $k \leq h$ if and only if 
 there is a set $\sigma$ in
 $\Sigma_h$ such that $i,j \in \sigma$. 
 We say that a vector $v_h$ in
 $\Bbbk^m$ \emph{realizes} $\Sigma_h$ if $v_{hi}
 = v_{hj}$ if and only if there exists a set $\sigma$ in $\Sigma_h$
 such that $i,j \in \sigma$. For instance, if $m=6$ and $\Sigma_h =
 \{\{1,5\},\{2\}, \{3,6\}, \{4\}\}$, then $(1,2,0,3,1,0)$ realizes
 $\Sigma_h$. When $v_h$ realizes $\Sigma_h$ we say that the
 \emph{type} of $v_h$ is $\Sigma_h$. Notice that if all points are distinct, then 
 $\Sigma_n = \{\{1\}, \ldots, \{m\}\}$.

\begin{lemma} \label{lemma:real}
 Let $P = \{p_1, \ldots, p_m\}$
be a set of distinct points in $\Bbbk^n$. Suppose that $\Bbbk$ contains at
 least $\binom{m}{2}+1$ elements. Then there exists $c_1, \ldots, c_{\overline{n}} \in \Bbbk$ such that all elements in the $m$-vector $c_{1} (p_{1i_1}, \ldots,
 p_{mi_1}) + \cdots + c_{\overline{n}} (p_{1 i_{\overline{n}}},
 \ldots, p_{m i_{\overline{n}-1}})$ are distinct, where $\overline{n} \leq \min(m,n)$ and $1=i_1 < i_2 < \cdots < i_{\overline{n}} \leq n$.
 The $c_i$'s 
 can be computed using 
$O(nm)$ comparisons and $O(\min(m,n)m^2 \log(m))$ operations provided a total order on $\Bbbk$. If $\Bbbk$ is not
ordered, then we need $O(nm)$ comparisons and $O(\min(m,n)m^4)$ operations.
 \end{lemma}
\begin{proof} Let $(i_1, \ldots, i_{\overline{n}})$ be the witness list with respect to $p_1, \ldots, p_m$. We will iterate through this list.  The idea is to start with a realization $v^{(h-1)} \in \mathbb{Z}^m$ of $\Sigma_{i_{h-1}}$ and show that 
$v^{(h)} = v^{(h-1)} + c_{h} (p_{1,i_h}, \ldots, p_{m,i_h})$ 
realizes $\Sigma_{i_h}$ for some $c_{h} \in \Bbbk$. The rest of the proof consists of an algorithm which we call the Distinct element algorithm.  
 \vspace{0.2cm}

 \noindent \textbf{Distinct element algorithm}


 \noindent
 At stage $1$, let $v^{(1)} = (p_{11},\ldots, p_{m1})$ and let $c_1 = 1$. \\
 \noindent
 At stage $h$, where $1<h \leq \overline{n}$, suppose that $v^{(h-1)}$ is a
 realization of $\Sigma_{i_{h-1}}$. Let $\Pi$
 be the set of pairs
 $(j,k)$ such that $j$ and $k$ are in different subsets of both
 $\Sigma_{i_{h-1}}$ and the type of $(p_{1 i_h}, \ldots, p_{m i_h})$.
Since $h-1 < \overline{n}$, the number of pairs $(j,k)$
 such that $j$ and $k$ are in different subsets of $\Sigma_{i_{h-1}}$
 is bounded by $\binom{m}{2}-1$. Thus $\Pi = \{(a_1,b_1), \ldots,
 (a_s,b_s)\}$ for some $s \leq \binom{m}{2} - 1$.
 Let
 $$\tau_j = \frac{v^{(h-1)}_{a_j} - v^{(h-1)}_{b_j}}{p_{b_j i_h}- p_{a_j i_h}} 
 \text{ for } j = 1, \ldots, s.$$ 
 Let $c_{h}$ be any non-zero element in $\Bbbk \setminus \{\tau_1,
 \ldots, \tau_s\}$ and let 
$$v^{(h)} = v^{(h-1)} + c_{h}
 (p_{1 i_h}, \ldots, p_{m i_h}).$$
 \noindent
At stage $\overline{n}+1$, we stop the algorithm and return $(c_1, \ldots, c_{\overline{n}})$.\\

\vspace{0.2cm}
\noindent \textbf{The correctness of the algorithm}\\
\noindent
We only need to show that $v^{(h)}$ realizes $\Sigma_{i_{h}}$. For $h=1$ this is clear. Pick
$a_j$ and $b_j$ in the same subset of $\Sigma_{i_{h}}$. Then $a_j$
and $b_j$ also belong to the same element in $\Sigma_{i_{h-1}}$. Hence, $v^{(h-1)}_{a_j} = v^{(h-1)}_{b_j}$ and $p_{a_j i_h} = p_{b_j i_h}$, 
from which it follows that $v^{(h)}_{a_j} = v^{(h)}_{b_j}$.

Suppose that $a_j$ and $b_j$ are in different subsets of
$\Sigma_{i_h}$. Suppose further that $a_j$ and $b_j$ are in the same
subset of $\Sigma_{i_{h-1}}$. Then $v^{(h-1)}_{a_j} = v^{(h-1)}_{b_j}$
and $p_{a_j i_h} \neq p_{b_j i_h}$ so that $v^{(h)}_{a_j} -
v^{(h)}_{b_j} = c_{h} p_{a_j i_h} - c_{h} p_{b_j i_h} \neq 0$. 

Thus,
it remains to check the case when $a_j$ and $b_j$ belong to different
sets in $\Sigma_{i_{h-1}}$. In this case, we have 
$v^{(h)}_{ a_j } - v^{(h)}_{b_j} = v^{(h-1)}_{a_j} - v^{(h-1)}_{b_j} + c_{h} p_{a_j i_h} - c_{h}
p_{b_j i_h}$. 
If $p_{a_j i_h} = p_{b_j i_h}$, then $v^{(h)}_{ a_j } -
v^{(h)}_{ b_j} = v^{(h-1)}_{a_j} - v^{(h-1)}_{b_j} \neq 0$. 
If $p_{a_j i_h} \neq p_{b_j i_h}$, then $v^{(h-1)}_{a_j} -
v^{(h-1)}_{b_j} + c_{h} p_{a_j i_h} - c_{h} p_{b_j i_h} = 0$
exactly when
$$c_h = \frac{v^{(h-1)}_{a_j} - v^{(h-1)}_{b_j}}{p_{b_j i_h}- p_{a_j i_h}}.$$
However, this can never occur as $p_{a_j i_h} \neq p_{b_j
 i_h}$ implies $(a_j,b_j) \in \Pi$ and $c_h$ was chosen to
differ from $\tau_j$. \vspace{0.2cm}

\noindent{\textbf{The complexity of the algorithm}}

\noindent Fix a stage $h>1$. We construct $\Pi$ as follows. Let $L_1$ be a list containing 
all pairs $(j,k)$ where $j$ and $k$ are in different subsets of $\Sigma_{i_{h-1}}$. Let $L_2$ be a list containing all pairs $(j,k)$ where $j$ and $k$ are of different subsets of the type of $(p_{1 i_h}, \ldots, p_{m i_h})$. Merge these lists into a new list $L$. 
The elements in $\Pi$ are exactly the elements which occur twice in $L$. After sorting $L$, we can easily obtain $\Pi$. Note that we 
use $O(m^2 \log(m^2)) = O(m^2 \log(m))$ elementary integer operations for this construction.
 
Constructing the list $(\tau_1, \ldots, \tau_s)$ from $\Pi$ requires $O(m^2)$ 
operations. 
If $\Bbbk$ is ordered, we sort the list using $O(m^2 \log(m))$ comparisons.
To find $c_h$, consider a list of $\binom{m}{2}$ nonzero elements in
$\Bbbk$. Take the first element in this list and check whether it is in 
$(\tau_1, \ldots, \tau_s)$. If it
is not, we are done. Otherwise, continue with the next element.
Finally, after at most $\binom{m}{2}$ checks, we will find an element which is not in $(\tau_1, \ldots,
\tau_s)$. Since $\Bbbk$ is ordered, each check requires 
$O(\log(m))$ comparisons. 

If $\Bbbk$ is not ordered, then we can not sort the list, so each check requires 
$O(m^2)$ comparisons.

Thus, we use mostly $O(m^2\log(m))$ 
comparisons if $\Bbbk$ is ordered and $O(m^4)$ comparisons otherwise.
Since we repeat the procedure $\overline{n}$ times, we are done with the complexity analysis. 
\end{proof}

\begin{theorem} \label{theorem:linearkbasisiso} Let $P = \{p_1, \ldots, p_m\}$
be a set of distinct points in $\Bbbk^n$. Suppose that $\Bbbk$ contains at
 least $\binom{m}{2}+1$ elements. We give an explicit isomorphism 
$S/I \cong k[x]/J$ and a $\Bbbk$-basis for $S/I$ of the form 
$\{[1],$ $[f], \ldots, [f^{m-1}]\}$, where $f$ is a linear form. The construction uses 
$O(nm)$ comparisons and $O(\min(m,n)m^2 \log(m))$ operations given a total order on $\Bbbk$. If $\Bbbk$ is not
ordered we need $O(nm)$ comparisons and $O(\min(m,n)m^4)$ operations.
\end{theorem}

\begin{proof}
 By Lemma \ref{lemma:real}, there exists $(c_1, \ldots, c_{\overline{n}})$ such that all elements in the $m$-vector $c_{1} (p_{1i_1}, \ldots,
 p_{mi_1}) + \cdots + c_{\overline{n}} (p_{1 i_{\overline{n}}},
 \ldots, p_{m i_{\overline{n}}})$ are distinct. If we let $x=c_{1} x_{i_1} + \cdots + c_{\overline{n}} x_{i_{\overline{n}}}$, it follows that $x(p_1), \ldots, x(p_m)$ are distinct. 
 Let $J$ be the ideal in
 $\Bbbk[x]$ vanishing on $x(p_1), \ldots, x(p_m)$. The the principal ideal $J$ is
 generated by $(x-x(p_1)) \cdots (x-x(p_m))$ and 
 a $\Bbbk$-basis for $\Bbbk[x]/J$ is the residues of $1, x, \ldots, x^{m-1}$. By Lemma \ref{lemma:iso}, $\Bbbk[x]/J \cong S/I$ and a $\Bbbk$-basis for $S/I$ can be chosen as the residues of $1, c_{1} x_{i_1} + \cdots + c_{\overline{n}} x_{i_{\overline{n}}}, \ldots, (c_{1} x_{i_1} + \cdots + c_{\overline{n}} x_{i_{\overline{n}}})^{m-1}$. 
 
The cost of the construction is dominated by the computation of the $c_i$'s, so the complexity result follows from Lemma \ref{lemma:real}.
 
 \end{proof}

\begin{example} \label{exampleiso} Consider the points $p_1, p_2,
 p_3,p_4,p_5,p_6$ from Example \ref{example1} as elements in $\mathbb{Z}_{43}$. 
 The witness list equals $\{1,2,4\}$, so we get $i_1 = 1, i_2 =
 2$ and $i_3=4$. The matrix describing the splittings is 
$$
\left( \begin{array}{cccccc|lc}
 1 & 1 & 1 & 0 & 0 & 2 & \Sigma_1 = & \{\{1,2,3 \}, \{4,5 \}, \{6 \} \} \\
 2 & 0 & 2 & 0 & 0 & 1 & \Sigma_2 = & \{\{1,3\}, \{2\}, \{4,5\}, \{6\} \} \\
 0 & 1 & 0 & 2 & 2 & 3 & \Sigma_3 = &\{\{1,3\}, \{2\}, \{4,5\}, \{6\} \} \\
 1 & 1 & 3 & 0 & 1 & 1 & \Sigma_4 = &\{\{1 \}, \{3 \}, \{2\}, \{4\}, \{5\}, \{6\} \}\\
\end{array} \right),
$$ 
where we omit the last two rows of the matrix.
The Distinct element algorithm is as follows. At stage 1: $v_1 =
(1,1,1,0,0,2)$. At stage 2, we see that the type of the second row
is $\{\{1,3\}, \{2,4,5\}, \{6\}\}$. The set of pairs built from
$\Sigma_{i_1}$ equals
$$\{(1,4), (1,5), (1,6), (2,4), (2,5), (2,6), (3,4), (3,5), (3,6), (4,6), (5,6)\}$$ and the set of pairs built from $\{\{1,3\}, \{2,4,5\}, \{6\}\}$ equals
$$\{(1,2), (1,4), (1,5), (1,6), (2,3), (3,4), (3,5), (3,6), (2,6), (4,6), (5,6)\}.$$
The intersection is equal to
$$\{(1,4), (1,5), (1,6), (2,6), (3,4), (3,5), (3,6), (4,6), (5,6)\}.$$
We compute $\tau_1$, which corresponds to the pair $(1,4)$ and equals
$\tau_1 = (v^{(1)}_1 - v^{(1)}_4)/(p_{4 i_2}- p_{1 i_2}) = (1 -
0)/(0-2) = -1/2 = 21$ and similarly for the other pairs to obtain
$$\{\tau_1, \ldots, \tau_9\} = \{21,21,1,42,21,21,42,41,41\}.$$
We choose $c_2 = 2$ to get $v_2^{(2)} = v_1^{(1)} + 2 (p_{1i_2}, \ldots,
p_{6i_2}) = (5,1,5,0,0,4)$, which is of type $\Sigma_{i_2}$ as
desired. At stage 3, the intersection equals
$$\{(1,4), (2,3), (2,4),(3,5), (3,6), (4,6)\}$$ and it turns out that we can use $c_3 = 1$ so that 
$v_3^{(3)} = v_2^{(2)} + (p_{1 i_3}, \ldots, p_{6 i_3}) = (6,2,7,0,1,5)$ which is
of type $\Sigma_{i_3} = \{\{1\}, \ldots, \{6\}\}$. The isomorphism
$$\mathbb{Z}_{43}[x]/(x(x-1)(x-2)(x-5)(x-6)(x-7)) \cong \mathbb{Z}_{43}[x_1, \ldots, x_6]/I$$
of algebras is induced by
$$x \mapsto c_1 x_{i_1} + c_2 x_{i_2} + c_3 x_{i_3} = x_1 + 2x_2 + x_4.$$ Thus, a vector space
basis for $\mathbb{Z}_{43}[x_1, \ldots, x_6]/I$ can be chosen as $$[1], [(x_1 +
2x_2 + x_4)], \ldots, [(x_1 + 2x_2 + x_4)^6].$$ Notice that $B(P)$ becomes a 
Vandermonde matrix and we have
$$
B = \begin{pmatrix}
1^0 	&1^0	&1^0	&1^0	&1^0	&5^0 \\
6^1 	&2^1	&7^1	&0^1	&1^1	&5^1 \\ 
6^2 	&2^2	&7^2	&0^2	&1^2	&5^2 \\
6^3 	&2^3	&7^3	&0^3	&1^3	&5^3 \\
6^4 	&2^4	&7^4	&0^4	&1^4	&5^4 \\ 
6^5 	&2^5	&7^5	&0^5	&1^5	&5^5
\end{pmatrix} \text{ and }
B^{-1}= 
\begin{pmatrix}
0 & 3 & 0 & 32 & 27 & 24\\
0 & 9 & 3 & 33 & 17 & 24\\
0 & 37& 37 & 15 & 10& 30\\
1 &7 &25 &12& 28&13\\
0&25& 28&8& 7& 19\\
0&5& 36& 29&40& 19
\end{pmatrix}.
$$
Let $f = x_1 x_2 x_4 + x_4 x_5 x_6x_7$. To compute the normal
form of $f$, we compute 
$f(P) = (2,0,24,0,0,2)$. By Lemma
\ref{lemma:nfwrtbasis}, we have
$$Nf(f,B) = ([e_1], \ldots, [e_m]) (B(P)^{-1})^t (f(p_1), \ldots, f(p_m))^t.$$
Since $(B(P)^{-1})^t (2,0,4,0,0,2)^t = (0,35,5,10,2,34)^t$ we conclude that 
$$Nf(f,B) = 35 (x_1 +
2x_2 + x_4) + 5(x_1 +
2x_2 + x_4)^2 + 10(x_1 +
2x_2 + x_4)^3$$ $$+ 2(x_1 +
2x_2 + x_4)^4 + 34 (x_1 +
2x_2 + x_4)^5.$$

Recall that there are closed expressions for the inverse of a Vandermonde matrix, so Gaussian elimination is not needed to compute 
$B^{-1}$. For reference, see for instance \cite{Klinger}.

\end{example}
\begin{remark}
 
In practice, the best way to obtain a realization vector $v$
corresponding to $\Sigma_{\overline{n}}$ is by nondeterministic
methods --- check if
$$c_{1} (p_{11}, \ldots, p_{m1}) + \cdots + c_{\overline{n}} (p_{1 i_{\overline{n}}}, \ldots, p_{m i_{\overline{n}}})$$
realizes $\Sigma_{\overline{n}}$, for some pseudo-random elements
$c_i$. This will be the case with probability close to one. If not, we
try with some other coefficients.
\end{remark}

\subsection{Construction 3 - Standard monomials with respect to the lexicographical order} \label{sec:mon} 
It was shown in \cite{CerliencoMureddu} that it is
possible to compute the set of standard monomials with
respect to the lexicographical order by purely combinatorial methods. The authors in \cite{CerliencoMureddu} presented an algorithm but did not make a complexity analysis of it. 
In \cite{Felszeghy}, it was indicated that the number of comparisons in a
straight forward implementation of the algorithm is proportional to
$n^2m^2$. One of the aims of the paper \cite{Felszeghy} was to improve the algorithm. This improved algorithm consists of three steps:
\begin{enumerate}
\item Construct the point trie $T_1$ with respect to $x_n, \ldots, x_1$.
\item Construct the lex trie $T_2$ from $T_1$ (see the Lex trie algorithm below). 
\item Return the set of standard monomials $\{x^{\alpha_1}, \ldots, x^{\alpha_m}\}$ 
with respect to the lexicographical ordering with $x_1 \succ \cdots \succ x_n$, where 
$\{\alpha_1, \ldots, \alpha_m\}$ is the set of paths from the root to the leaves in $T_2$. 
\end{enumerate}
Note that the associated point trie is built backwards, i.e. we read the coordinates of the points from right to left.
In the complexity analysis in \cite{Felszeghy}, it was shown that the first step
requires $O(nmr)$ comparisons, where we recall that $r$ denotes 
the maximal number of edges from a vertex in the trie.
The second step requires 
$O(nm)$ integer summations, and the third step
requires $O(nm)$ readings of integers bounded by $m$. In total, the construction is dominated by 
$O(nmr)$ comparisons. We now improve this.
\begin{theorem} \label{thm:lextrie}
We can compute the set of standard monomials with respect to the lexicographical order using
$O(nm +m\min(m,nr))$ comparisons.
\end{theorem}

\begin{proof} 
By Theorem \ref{thm:cmpalg}, it is possible to construct the associated point trie using 
$O(nm +m\min(m,nr))$ comparisons. Hence the result follows by using the construction of the standard monomials from the associated point trie as given in 
\cite{Felszeghy}.
\end{proof}

For completeness, we will state the algorithm that construct the standard monomials from the point trie. 
For a proof, we refer the reader to \cite{Felszeghy}. We remark that our formulation is in terms of the $\Sigma_i$'s.

\vspace{0.2cm}
\noindent \textbf{Lex trie algorithm (\cite{Felszeghy})}\\
\noindent 
Fix some stage $h>0$. Let $v_0, \ldots, v_j$ be the set of vertices on level $h$ of the trie 
(at the root level $1$, $v_0 = \{1, \ldots, m\}$). 
For an arbitrary equivalence class $\{i_1, \ldots, i_k\}$ in $\Sigma_{n-h}$, we
let $v_{a,b} = v_{a,b} \cup \{i_k\}$ if $i_k \in v_a$ and exactly $b$ elements in 
$\{i_1, \ldots, i_{k-1}\}$ also belong to $v_a$. (Initially $v_{a,b}$ is empty.)
The vertex set at the $(h+1)$-th level of the trie consists of all nonempty $v_{a,b}$. 
If $v_{a,b}$ is nonempty, there is an edge between $v_{a}$ and $v_{a,b}$ which is labelled by $b$. 
\qed
\vspace{0.2cm}
\noindent 
Since the paper \cite{Felszeghy} does not contain a full example illustrating this algorithm, we give such an example here. 

\begin{example}
Let $p_1 = (1,0,2,1), p_2 = (1,1,0,1), p_3 = (3,0,2,1)$, $p_4 = (0,2,0,0)$, $p_5 = (1,2,0,0)$ and 
$p_6 = (1,3,1,2)$. Suppose that we want to compute the standard monomials of $I(P)$ with respect to the lexicographical order with $x_1 \succ \cdots \succ x_n$. First we have to construct the point trie with respect to the points read from right to left. 
The example is constructed to give the first four rows
from Example \ref{example1}, that is 
\begin{displaymath}
\left( \begin{array}{cccccc|lc}
 1 & 1 & 1 & 0 & 0 & 2 & \Sigma_1 = & \{\{1,2,3 \}, \{4,5 \}, \{6 \} \} \\
 2 & 0 & 2 & 0 & 0 & 1 & \Sigma_2 = & \{\{1,3\}, \{2\}, \{4,5\}, \{6\} \} \\
 0 & 1 & 0 & 2 & 2 & 3 & \Sigma_3 = &\{\{1,3\}, \{2\}, \{4,5\}, \{6\} \} \\
 1 & 1 & 3 & 0 & 1 & 1 & \Sigma_4 = &\{\{1 \}, \{3 \}, \{2\}, \{4\}, \{5\}, \{6\} \}\\
\end{array} \right).
\end{displaymath}
\begin{figure}[ht!] \label{lextrie} \centering
\includegraphics{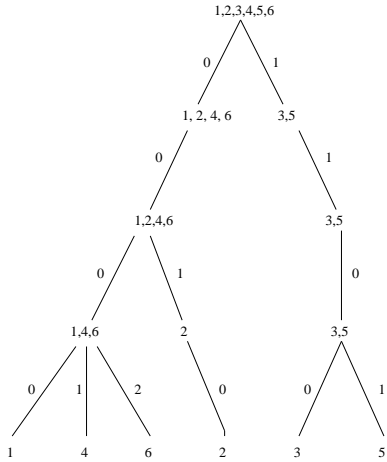}
 \caption{The associated lex trie}
\end{figure}
Figure 2 
shows the lex trie. The paths from the root to the leaves equals 
$(0,0,0,0)$, $(0,0,0,1)$,$(0,0,0,2)$, $(0,0,1,0)$, $(1,0,0,0)$ and $(1,0,0,1)$. It follows that the 
set of standard monomials is equal to $\{1,x_4,x_4^2, x_3, x_1,x_1x_4\}.$

We will now describe the creation of the lex trie stage by stage.

\begin{itemize}
 \item[Stage $1$]
There is only one vertex on the root level; $v_0 = \{1,2,3,4,5,6\}$. We iterate over the equivalence classes of
$\Sigma_{4-1} = \Sigma_3 = \{\{1,3\}, \{2\}, \{4,5\}, \{6\} \}$. We start with 
$\{1,3\}$ and get $v_{00} = \{1\}$ and $v_{01} = \{3\}$ 
For $\{2\}$ we get
$v_{00} = \{1,2\}$. For $\{4,5\}$ we get $v_{00} = \{1,2,4\}$ and $v_{01} = \{3,5\}$. Finally, for
$\{6\}$ we get $v_{00} = \{1,2,4,6\}$.
 \item[Stage $2$]
We rename the two vertices from the previous stage to $v_{0} = \{1,2,4,6\}$ and $v_{1} = \{3,5\}$. 
We start by iterating on the equivalence classes of $\Sigma_{4-2} = \Sigma_2 = \{\{1,3\}, \{2\}, \{4,5\}, \{6\} \}$.
For $\{1,3\}$ we get $v_{00} = \{ 1 \}$ and $v_{10} = \{ 3 \}$, since $1 \in v_0$ and $3 \in v_1$. For $\{2\}$ we get
$v_{00} = \{ 1,2 \}$, since $2 \in v_0$. For $\{4,5\}$, we get $v_{00} = \{1,2,4\}$ and $v_{10} = \{3,5\}$, 
and for $\{6\}$ we get $v_{00} = \{1,2,4,6\}$.
 \item[Stage $3$]
We rename the two vertices from the previous stage to $v_{0} = \{1,2,4,6\}$ and $v_{1} = \{3,5\}$.
On $\Sigma_{4-3} = \Sigma_1$, we begin with $\{1,2,3 \}$. Since $1 \in v_0$, we let
$v_{00} = \{1\}$. Since also $2 \in v_0$, we put $2$ in $v_{01}$. We have then 
$3 \in v_{1}$, thus $v_{10} = \{3\}$. We continue with
$\{4,5 \}$. Since $4 \in v_{0}$, we let $v_{00} = \{1,4\}$ and since $5 \in v_{1}$, we let $v_{10} = \{3,5\}$. 
Finally, for $\{6\}$ we get $v_{00} = \{1,4,6\}$.
 \item[Stage $4$]
We rename the vertices from the previous stage to $v_{0} = \{1,4,6\}$, $v_{1} = \{2\}$ and $v_{3} = \{3,5\}$.
In $\Sigma_{4-4} = \Sigma_0$ there is only one equivalence class: $\{1,2,3,4,5,6\}$. 
We have: 
$1 \in v_{0} \Rightarrow v_{00} = \{1\}$,
$2 \in v_{1} \Rightarrow v_{10} = \{2\}$,
$3 \in v_{2} \Rightarrow v_{20} = \{3\}$,
$4 \in v_{0} \Rightarrow v_{01} = \{4\}$,
$5 \in v_{3} \Rightarrow v_{21} = \{5\}$ and
$6 \in v_{0} \Rightarrow v_{02} = \{6\}.$

\end{itemize}
\end{example}

\subsection{Construction 4 - Standard monomials for some elimination orders} \label{sec:conssmelim}

In this section, we will 
study a method that can be used to create standard monomials with respect to some 
elimination orders. An elimination order $\prec$ with respect to the variables $x_1, \ldots, x_{i-1}$ is an admissible monomial order on $S$ satisfying the condition
$$\ini(f) \in \Bbbk[x_i, x_{i+1}, \ldots, x_{n}] \Rightarrow f \in \Bbbk[x_{i},x_{i+1}, \ldots, x_{n}]. $$ 
We will show that using $O(nm)$ comparisons and $O(\min(m,n)m^3)$ 
 operations we can compute the set of standard monomials with respect to some 
elimination orders and, thus, for these orders, our method has 
better asymptotic behavior than the Buchberger-M\"oller algorithm when $m<n$.

A classic example of an elimination order is the lexicographical order 
with respect to $x_1 \succ \cdots \succ x_n$ (However, for 
the lexicographical order, we already have a fast method to compute the standard monomials.) 

We will construct an elimination order from two partial orders as follows. Let $\prec_1$ be any admissible monomial order on $\{x_{1},x_{2}, \ldots, x_{i-1}\}$ and let $\prec_2$ 
be any admissible monomial order on $\{x_{i},x_{i+1}, \ldots, x_{n}\}$. Now define $\prec$ 
by $x^{\alpha} \prec x^{\beta}$ if 
$x_i^{\alpha_i} \cdots x_n^{\alpha_n} \prec_2 x_i^{\beta_i} \cdots x_n^{\beta_n}$ or
$x_i^{\alpha_i} \cdots x_n^{\alpha_n} = x_i^{\beta_i} \cdots x_n^{\beta_n}$ and
$x_1^{\alpha_1} \cdots x_{i-1}^{\alpha_{i-1}} \prec_1 x_1^{\beta_1} \cdots x_{i-1}^{\beta_{i-1}}$.
The order $\prec$ then becomes an elimination order with respect to $\{x_{1},x_{2}, \ldots, x_{i-1}\}$.

Let $\{p_1, \ldots, p_m\}$ be a set of distinct
points in $\Bbbk^n$. Let $\tau$ be a permutation of $\{1, \ldots, n\}$. 
Define $\tau(p_i) = (p_{i\tau(1)}, \ldots, p_{i\tau(n)})$. Simplified, this means that $\tau(p_i)$ is 
$p_i$ with respect to the coordinates $x_{\tau(1)}, \ldots, x_{\tau(n)}$. 
Let $W_{\tau} = \{i_1, \ldots, i_{\overline{n}}\}$ be the witness list derived from the
$\Sigma$-algorithm with respect to $\tau(p_1), \ldots, \tau(p_n)$. 
By Lemma \ref{lemma:diffpoints}, the points
$q_1, \ldots, q_m$ are distinct, where $q_j = \pi(\tau(p_j))$ and 
$\pi((a_1, \ldots,a_n)) = (a_{i_1}, \ldots, a_{i_{\overline{n}}}).$
Let $W_{\tau}^c = \{j_1, \ldots, j_{\underline{n}}\}$ be the complementary set to 
$W_{\tau}$.
Let $\prec_1$ be any admissible monomial order on $x_{j_1}, \ldots, x_{j_{\underline{n}}}$. Let
$\prec_2$ be any admissible monomial order on $x_{i_1}, \ldots, x_{i_{\overline{n}}}$. Let $B$ be the 
set of standard monomials with respect to $\prec_2$. 
If $\prec$ is the elimination order constructed from $\prec_1$ and $\prec_2$, 
then it is clear that $B$ is the set of standard monomials with respect to $\prec$ as well.
We determine $B$ by the combinatorial algorithm in the case when $\prec_2$ is the lexicographical order. When 
$\prec_2$ is another order, we use the Buchberger-M\"oller algorithm. Except for complexity issues, we have proved
the following theorem.

\begin{theorem} \label{thm:elim}
Using the notation above, for any permutation $\tau$ and for any
elimination order $\prec$ with respect to $x_{j_1}, \ldots, x_{j_{\underline{n}}}$, we can
determine the set $B$ of standard monomials with respect to $\prec$ using
at most $O(nm)$ comparisons and $O(\min(m,n)m^3)$ operations. 
If $\prec$ is the lexicographical order on $x_{i_1}, \ldots, x_{i_{\overline{n}}}$, then we determine
$B$ using at most $O(nm+m\min(m,nr))$ comparisons.

In both cases, there is a normal form method
which can be initiated using $O(m^3)$ operations. The normal form method uses at most $O(|f|sm+m^2)$ 
operations for a polynomial $f$ with $s$ monomials.
\end{theorem}

\begin{proof}
By Theorem \ref{thm:cmpalg}, it takes $O(nm +m\min(m,nr))$ comparisons to preprocess the points. The call to the Buchberger-M\"oller algorithm uses $O(\overline{n}m^2 + \min(m,\overline{n})m^3)$ operations by Theorem 2 in \cite{lundqvist}. Since $\overline{n} \leq \min(m,n)$, the call requires at most $O(\min(m,n)m^3)$ operations. 
On the other hand, if we use the combinatorial method in the case when $\prec_2$ is the lexicographical order, we would get $O(nm + m\min(m,nr))$ comparisons as bound. 

Hence, to compute $B$, we use at most $O(m^4)$ operations and $O(nm)$ comparisons when $\prec$ is not the lexicographical order. When 
$\prec$ is the lexicographical order, we use $O(nm +m\min(m,nr))$ comparisons. The statement about
the normal form method follows from Proposition \ref{prop:normalformPoints}. 

\end{proof}

\begin{example}

Let $p_1, \ldots, p_6$ be the points in Example \ref{example1}. Let $\tau$ be the permutation
$$
\begin{pmatrix} 
1& 2& 3& 4& 5& 6& 7& 8 \\
8& 2& 3& 4& 6& 5& 7& 1
\end{pmatrix}
$$

To simplify notation, let $y_i = x_{\tau(i)}$. We compute the 
$\Sigma_i$'s with respect to $y_8, y_7,  \ldots, y_1$. Since 
$\Sigma_3$ equals $\{\{1 \}, \{2 \}, \{3\}, \{4\}, \{5\}, \{6\} \}$, we only display the first three rows.

 \begin{displaymath}
\left( \begin{array}{cccccc|lc}
 1 & 1 & 1 & 0 & 0 & 2 & \Sigma_1 = & \{\{1,2,3 \}, \{4,5 \}, \{6 \} \} \\
 3 & 3 & 2 & 2 & 2 & 2 & \Sigma_2 = &\{\{1,2 \}, \{3 \}, \{4,5\}, \{6\} \}\\
 1 & 2 & 3 & 4 & 5 & 6 & \Sigma_3 = &\{\{1 \}, \{2 \}, \{3\}, \{4\}, \{5\}, \{6\} \}
\end{array} \right).
 \end{displaymath}
 
Thus, $W_1 = \{1,2,3\}$ and $S/I$ is isomorphic to 
$\Bbbk[y_6,y_7,y_8]/I(Q)$, where $Q = \{(1,3,1), (1,3,2), \ldots, (2,2,6)\}$. 
Let $\prec_2$ be the degree reverse lexicographical order with 
$y_6 \succ_2 y_7 \succ_2 y_8$. A call to the Buchberger-M\"oller algorithm gives $B_{\prec_2} = \{1,y_8,y_7,y_6,y_8^2,y_6y_8\}$. Let $\prec_1$ be the lexicographical order
on $\Bbbk[y_1,\ldots,y_5]$ with $y_1 \succ_1 \cdots \succ_1 y_5$
and construct $\prec$ from $\prec_1$ and $\prec_2$ above. Then $B_{\prec} = \{1,y_8,y_7,y_6,y_8^2,y_6y_8\}$.

Suppose that we want to compute the normal form of $f = y_8 y_2 y_4 + y_4 y_6 y_5 y_7$.
We determine $B(\tau(P))^{-1}$ and $(\tau(f(p_{1})), \ldots, \tau(f(p_6))) =(2,0,24,0,0,2).$ 
Notice that $\tau(f(p_{1i})) = g(p_i)$ where $g = x_1 x_2 x_4 + x_4 x_5 x_6 x_7$. 
The normal form will be
$$(1, y_8, y_7,y_6,y_8^2,y_6y_8)(B(\tau(P))^{-1})^t (\tau(f(p_1)), \ldots, \tau(f(p_m)))^t$$ $$=
12 + 18y_8 + 37 y_7 + 35 y_8^2 + 41 y_6y_8.$$

In terms of the $x_i$'s, the order $\prec$ is the degree reverse lexicographical order on $\Bbbk[x_5, x_7,x_1]$ with
$x_5 \succ x_7 \succ x_1$ and it is the lexicographical order on $\Bbbk[x_8, x_2,x_3,x_4,x_6]$ with 
$x_8 \succ x_2 \succ x_3 \succ x_4 \succ x_6$. The normal form of $x_1 x_2 x_4 + x_4 x_5 x_6 x_7$
is $12 + 18x_1 + 37 x_7 + 35 x_1^2 + 41 x_1x_5$.

\end{example}

\section{Biological implications} \label{sec:bio}

In the algebra approach to reverse engineering, 
we have some experimental data $S = \{s_1, \ldots s_{m+1}\}$, where each $s_i$ is a realvalued vector of size $n$ and
$m \ll n$.
Each $s_{ij}$ is then discretized into a prime number $p$ of states so that
the $s_i$'s can be viewed as elements in $\mathbb{Z}_p^n$. 
For the discretized data, we want to find transition functions $h_1, \ldots, h_{n}$ such that
$h_i(s_j) = s_{j+1,i}$ for $j =1, \ldots, m$. Finally, we wish to find the normal form of 
the $h_i$'s with respect to a set $B_{\prec}$ of standard monomials for $I(S)$, 
for some admissible monomial order $\prec$ on $\mathbb{Z}_p[x_1, \ldots, x_n]$.

It is easily seen that 
\begin{equation} \label{eq:timeseries} 
h_i = s_{2,i} f_1 + \cdots + s_{m+1,i} f_{m},
\end{equation}
where $f_1, \ldots, f_{m}$ is a set of separators with respect to $s_1, \ldots, s_{m}$. 

An example illustrating the method in \cite{LaubStigler04} is given in the same paper: 
After discretizing over $\mathbb{Z}_3$, one has
$$s_1 = (2,2,2), s_2 = (1,0,2), s_3 = (1,0,0), s_4 = (0,1,1), s_5 = (0,1,1).$$ For computing the 
$h_i$'s, the authors use an $O(n^2m^2)$ algorithm to get
$h_1 = x_1^2 x_3 +2 x_1^2 + x_1x_3 + x_1$, 
$h_2 = 2x_1^2x_3 + x_1^2 +2 x_1x_3 +2 x_1 +1$,
$h_3 = 2x_1^2x_3 +2 x_1^2 +2 x_1x_3 + x_1 + 1$.

Then, the lexicographical order with respect to $x_1 \succ x_2 \succ x_3$ is used to determine a 
Gr\"{o}bner basis for $I(\{s_1,s_2,s_3,s_4\})$ by means of the Buchberger-M\"oller algorithm. This Gr\"{o}bner basis is being equal to 
$\{x_1+x_2+2,x_2x_3 + x_2 + 2x_3^2 +2 x_3, x_2^2+x_2 +2 x_3^2 +2 x_3\}$.
Finally, the $h_i$'s are reduced using the Gr\"{o}bner basis and we get
$$\nf(h_1,B) = -x_3^2 + x_3, \nf(h_2,B) = x_3^2 - x_3+1,\nf(h_3,B) = -x_3^2 + x_2+1,$$ 
where $B$ denotes the complement to the initial ideal with respect to the chosen order.
Since the normal form is computed by means of the reduction with respect to the Gr\"obner basis, and not by means of Lemma \ref{lemma:nfwrtbasis}, the worse time 
complexity for this part for general $m$ and $n$ is reported by the authors to be $O(n(m-1)2^{cm+m-1})$, using the bound in \cite{Dube}.

Our approach is the following. Firstly, we determine the set of standard monomials of 
$I(S)$ using the 
Lex trie algorithm to get $B = \{1,x_3,x_3^2,x_2\}$. The equality $s_4 = s_5$ is detected during the 
$\Sigma$-algorithm and we need to use only $P=\{s_1, \ldots, s_4\}.$ 
We get 
$$B(P) = 
\begin{pmatrix}
1 & 1 & 1 & 1\\
2 & 2 & 0 & 1\\
1 & 1 & 0 & 1\\
2 & 0 & 0 & 1
\end{pmatrix}
\text{ and } 
B(P)^{-1} = 
\begin{pmatrix}
0&2&2&2\\
0&2&0&1\\
1&0&2&0\\
0&2&2&0
\end{pmatrix}.
$$
By Lemma \ref{lemma:sepfrombasis}, 
$$(f_1, f_2, f_3, f_4)^t = B(P)^{-1} (1,x_3,x_3^2,x_2)^t$$ 
$$= (2x_3 + 2x_3^2 + 2x^2, 2x_3 + x_2, 1 + 2x_3^2, 2x_3 + 2x_3^2)^t.$$
Finally, we determine the normal forms of the $h_i$'s in terms of the elements in $B$ by using (\ref{eq:timeseries}):
$$\nf(h_1,B) = 1 f_1 + 1 f_2 + 0 f_3 + 0 f_4 = x_3 + 2 x_3^2$$ and similarly 
$$\nf(h_2,B) = 1 + 2x_3 + x_3^2 \text{ and } \nf(h_3,B) = 1 +2x_3^2 + x^2.$$

We will show below that for determining the $h_i$'s, it is in general enough to use $O(m^3)$ operations. 
In total, our approach uses $O(nm)$ comparisons and $O(m^3)$ operations, which is drastically better than
the exponential algorithm involving reduction with respect to the Gr\"{o}bner basis. Notice that we do not
use the construction of the separators from Section \ref{sec:conssep}. That is due to our wish to write them as linear 
combinations of elements in $B$. 

We are aware of the fact that lexicographical order is not always the best choice for a monomial order in these applications. 
Since $m \ll n$, almost any admissible monomial order will be an elimination order and
we believe that the method in Section \ref{sec:conssmelim} can be used in a lot of cases 
to find a feasible order. 

To complete the section, we state the following general theorem, where the complexity results are written assuming that $m<n$.

\begin{theorem}
The number of arithmetic operations in our approach to the reverse engineering method presented in 
$\cite{LaubStigler04}$ is
$O(nm)$ comparisons and $O(m^3)$ operations when $\prec$ is the lexicographical order. For the elimination orders in 
Section \ref{sec:conssmelim}, a bound for the number of comparisons is $O(nm)$ and a bound for 
the number of operations is $O(\min(m,n)m^3)$. For an arbitrary order, the
number of operations is bounded by $O(nm^2 + \min(m,n)m^3)$.
\end{theorem}
\begin{proof}
Except for determining the standard basis $B$ and inverting $B(S)^{-1}$, 
the only necessary computation is determination of the $h_i$'s. By (\ref{eq:timeseries}) and Lemma \ref{lemma:sepfrombasis} we have
$$h_i = (s_{2,i}, \ldots, s_{m+1,i}) B^{-1} (e_1, \ldots, e_m)^t.$$ 
Thus, computation
of each $h_i$ requires $O(m^2)$ operations and to compute all $h_i$'s requires 
$O(m^3)$ arithmetic operations. 

The theorem for the lexicographical case follows from 
Theorem \ref{thm:lextrie}. In the elimination case the result follows from Theorem \ref{thm:elim} 
and in the general case it follows from Theorem 2 in \cite{lundqvist}, where it is shown that the 
Buchberger-M\"oller algorithm
uses at most $O(nm^2 + \min(m,n)m^3)$ operations.
\end{proof}

\section{Acknowledgment}

The author would like to thank Clas L\"ofwall for valuable comments on the paper and J\"orgen Backelin for presenting useful ideas in the proof of 
Proposition \ref{prop:sigmaalg}.

\appendix
\section{Appendix} \label{app}
We will describe two combinatorial algorithms to perform preprocessing of
the points. While the first algorithm iterates over the coordinates, the second one, given in \cite{Felszeghy}, iterates over the points. 
However, the two algorithms turn out to perform exactly the same comparisons and we prove that the number
of comparisons that are needed is bounded by $nm + m\min(m,nr)$. 

\subsection{Preprocessing by iterating over the coordinates} \label{subsec:sigma}
Here we present an algorithm which computes the $\Sigma_i$'s, the witness list and the witness matrix by iterating 
over the coordinates. The number of operations of this algorithm is
bounded by $O(nm + m^2)$. We will sharpen the number of operations in Section \ref{section:compare}.

\begin{proposition} \label{prop:sigmaalg}
Let $\Omega$ be a set equipped with an equivalence relation.
Let $v_1, \ldots, v_m$ be $n$-tuples of elements in $\Omega$.
 The $\Sigma$-algorithm below computes the equivalence classes $\Sigma_1, \ldots, \Sigma_n$ using
at most $n m + m^2$ $\Omega$-comparisons. 

\end{proposition}
\begin{proof}
The proof consists of three parts.
\vspace{0.2cm}

\noindent{\textbf{Formulation of the $\Sigma$-algorithm }}

\noindent 
At stage $0$, let $W_0 = \{ \}$ and let $C$ consist of zero entries only. \\
At stage $h$, let $W_h$ be the witness list with respect to
$\pi_h(v_1), \ldots, \pi_h(v_m)$, let $C$ be the witness matrix with respect to
$\pi_h(v_1), \ldots, \pi_h(v_m)$ and let 
$\Sigma_{h} = \{\Sigma_{h,1}, \ldots, \Sigma_{h,k}\}$.

\noindent
At stage $h+1$, let $\Undone$ be the set of elements in $\Sigma_h$
which contains at least two elements. Let $\Done$ be the set of
elements in $\Sigma_h$ which contains only one element. Proceed as
follows. 

Pick a set $T$ from $\Undone$ and let $\Undone = \Undone
\setminus \{T\}$. Let $i$ be the first element in $T$ and let $T_1$
be the set of indices $j$ in $T$ for which $v_{i,h+1} \neq v_{j,h+1}$.
Let $T_2 = T \setminus T_1$ and $\Done = \Done \cup \{T_2\}$. If $T_1$
contains exactly one element, let $\Done = \Done \cup \{T_1\}$. If
$T_1$ contains more than one element, let $\Undone = \Undone \cup
\{T_1\}$. Also set $c_{\min(i,j),\max(i,j)} = h+1$, for all $i \in T_2$ and all
$j \in T_1$. Repeat until $\Undone$ is empty and finally let
$\Sigma_{h+1} = \Done$. If at least one $T_1$ was non-empty during the
process, let $W_{h+1} = W_h \cup \{h+1\}$. Otherwise, let $W_{h+1} =
W_h.$ $\Undone$ will eventually get empty since we remove a set $T$
from $\Undone$ in each step described above, and in the cases when we
insert an element, this element will have lower cardinality compared
to the set that we removed.
\noindent
We stop the algorithm either when $\Sigma_{h+1}$ contains no elements with
more than one element or after performing the $n$th step. If we stop at stage
$h+1$, we set $W = W_{h+1}$ and $\Sigma_{i} = \Sigma_{h+1}$ for $i = h+1, \ldots, n$. 
\vspace{0.2cm}

\noindent{\textbf{The correctness of the algorithm}}

\noindent Clearly $\Sigma_0 = \{\{1,2,\ldots, m\}\}$ and $W_0 = \{ \}$
agree with the assumptions made at stage $h=0$. After stage $h+1$, we
also see that $\Sigma_{h+1}$ contains disjoint subsets. Suppose that
we pick two elements $i$ and $j$ from different subsets of
$\Sigma_{h+1}$ and suppose that $v_{ik} = v_{jk}$ for $k \leq h$.
Then $i$ and $j$ are in the same equivalence class of $\Sigma_h$, and are
splitted in stage $h+1$. Thus $v_{i,h+1} \neq v_{j,h+1}$ and
$c_{ij} = h+1$. Suppose that we pick two elements $i$ and $j$ from
the same subset. Then $v_{i,h+1} = v_{j,h+1}$ and by assumption, $v_{ik} = v_{jk}$ for $k \leq h$, so we conclude that $v_{ik} = v_{jk}$ 
for all $k \leq h+1$ and, hence, $c_{ij} = 0$. Thus the assumptions
made at stage $h$ hold also for stage $h+1$. Besides from the
complexity, the correctness of the proposition
now follows by performing the algorithm to stage $n$. 
\vspace{0.2cm}

\noindent{\textbf{The complexity of the algorithm}}

\noindent We will split a set $T$ into two sets $T_1$ and $T_2$, were
$T_1$ is non-empty, at most $\overline{m}-1$ times. Every time we split,
we will perform at most $m$ comparisons, resulting
in a bound of $m^2$ comparisons. At each stage there are at
most $m$ comparisons resulting in a non-splitting. At most $n$ times
we will not perform a splitting of $T$. Still, we need to perform
$m$ comparisons in order to make sure we do not need to split,
resulting in a bound of $nm$ comparisons.
Thus, the overall upper bound is $nm + m^2$ comparisons. 

\end{proof}

\begin{remark}
It is clear that the number of elementary integer operations for the construction of 
for $W$ and $C$ is $O(\overline{n})$ and $O(m^2)$ respectively.
\end{remark}

Suppose that $\Omega$ is equipped with a total order and suppose that 
$v_1, \ldots, v_m$ are sorted lexicographically with respect to this order. Then, if
$v_{h,j} \neq v_{h,j+1}$, we know that $v_{h,j} \neq v_{h,j+2}$, $v_{h,j} \neq v_{h,j+3}$ and so on. 
For every stage $h$, we need to perform at most $m$ comparisons, thus we can 
compute the $\Sigma_i$'s using only $O(nm)$ comparisons. Since the complexity of
sorting $v_1, \ldots, v_m$ is $O(nm\log(m))$, we have the following proposition.


\begin{proposition} \label{prop:sortalg}
Suppose that $\Omega$ is equipped with a total order. The $\Sigma_i$'s can be computed 
using $O(nm)$ comparisons if the $v_i$'s are sorted and 
$O(nm\log(m))$ comparisons otherwise.
\end{proposition}

\begin{remark}
In the case when $\Omega = \mathbb{Z}_p$ and $p$ is less than $\log(m)$, it may be useful to sort the vectors using 
bucket sort instead of the classical merge sort. The complexity for the sorting step then becomes  $O(nmp)$. 
\end{remark}

\subsection{Preprocessing by iterating over the vectors}
The second algorithm was first formulated in terms of the point trie \cite{Felszeghy}. 
The algorithm iterates over the vectors, assuming that a partial trie exists and inserts a vector into it. 
The number of comparisons that are needed was reported to be bounded by $nmr$, 
where we recall that $r$ denotes the maximal number of edges from a vertex in the tree. 
In Section \ref{section:compare} we will sharpen this bound.
We formulate the Point trie algorithm in terms of the $\Sigma_i$'s to simplify the comparison of the algorithm with the $\Sigma$-algorithm. We omit the bookkeeping of the witness list and the witness matrix since this
was not part of the original Point trie algorithm. 

\vspace{0.5cm}

\noindent \textbf{The Point trie algorithm}

\noindent 
Let $\Omega$ be a set equipped with an equivalence relation and let 
$v_1, \ldots, v_m$ be elements in $\Omega^n$.

At stage $h$, suppose that $T^{(h)}$ is a trie with respect to $v_1, \ldots, v_h$. Let $\Sigma^{(h)}_1, \ldots, \Sigma^{(h)}_n$ be the equivalence classes defining $T^{(h)}$.  To construct the trie $T^{(h+1)}$, check if $v_{h+1,1} = v_{i,1}$ for some 
$i \leq h$ by iterating over the equivalence classes in $\Sigma^{(h)}_1$. 

\begin{itemize}
\item
If it did not, let 
$\Sigma^{(h+1)}_i = \Sigma^{(h)}_i \cup \{\{h+1\}\}$ for $i =1, \ldots, n$ and stop.
\item
If it did, then  $i \in \Sigma^{(h)}_{1,j}$ for some $j$. Let 
$\Sigma^{(h+1)}_{1,k} = \Sigma^{(h)}_{1,k}$ for $k \neq j$ and let 
$\Sigma^{(h+1)}_{1,j} = \Sigma^{(h)}_{1,j} \cup \{h+1\}$. Continue to check if 
$v_{h+1,2} = v_{i,2}$ for some $i \in  \Sigma^{(h)}_{1,j}$, by iterating over the children of 
$\Sigma^{(h)}_{2}$. 
\begin{itemize}
\item
If it did not, let 
$\Sigma^{(h+1)}_i = \Sigma^{(h)}_i \cup \{\{h+1\}\}$ for $i =2, \ldots, n$ and stop.
\item
If it did, then $i \in \Sigma^{(h)}_{2,j}$ for some $j$. Let 
$\Sigma^{(h+1)}_{2,k} = \Sigma^{(h)}_{2,k}$ for $k \neq j$ and let 
$\Sigma^{(h+1)}_{2,j} = \Sigma^{(h)}_{2,j} \cup \{h+1\}$. Continue in the same fashion.
\end{itemize}
\end{itemize}
\qed

\begin{remark} \label{remark:orderpoints}
It was indicated in \cite{Felszeghy} that the assumption of $\Omega$ being equipped with a total order 
makes it possible to create the point trie by iterating over the points 
using $O(nm\log(m))$ comparisons. We do not agree with 
the argument given in \cite{Felszeghy}. It would prove that insertion sort has complexity 
$O(m\log(m))$ (the algorithm is insertion sort when $n = 1$), which is a contradiction. It is not hard to prove that the 
correct bound should read $O(m^2 + nm\log(m))$ for this method. 
However, if we assume that the points are sorted from the beginning, then we can manage in $O(nm)$ operations. 
\end{remark}

\subsection{Comparing the preprocessing algorithms} \label{section:compare} 

We developed the $\Sigma$-algorithm, described in \ref{subsec:sigma}, as an effective way to build the witness matrix and the witness list and we had the 
constructions from Section \ref{sec:conssep} and Section \ref{sec:conssmelim} in mind. 
Later we realized that the $\Sigma_i$'s could be used in Section \ref{sec:consiso}. 
When reading the paper \cite{Felszeghy}, we understood that the 
$\Sigma$-algorithm could also be used to improve the combinatorial computations of standard monomials with respect to the lexicographical order. 
It turned out that it was not obvious that our method 
was to prefer. Indeed, when $nr \ll m$, the Point trie algorithm seemed to have better asymptotic behavior. 
This lead us to make an extensive comparison of the algorithms, 
and it turned out that the two algorithms perform the same 
comparisons! A corollary to this is that both algorithms share the upper bound $O(nm+m\min(m,nr))$. 

\begin{proposition} \label{prop:equalcmp}
The $\Sigma-$algorithm and the Point trie algorithm perform the same comparisons.
\end{proposition}

\begin{proof}

Consider first the $\Sigma$-algorithm.
Fix an element $v_{ij}$, with $i \in \Sigma_{j-1,k}$. Then 
$v_{ij}$ is compared only with elements $v_{hj}$, with $h \in \Sigma_{j-1,k}$. More precise, if $h<i$,
then $v_{ij}$ is compared with $v_{hj}$ if $v_{h'j} \neq v_{hj}$ for $h' < h$. If $h>i$, then 
$v_{ij}$ is compared with $v_{hj}$ if $v_{h'j} \neq v_{hj}$ for all $h' < i$.

Now consider the Point trie algorithm. Upon insertion of the vector $v_i$ into the partial trie $T^{(i-1)}$, the element $v_{ij}$ is compared with elements $v_{hj}$, with $h \in \Sigma_{j-1,k}^{i-1}$. But $v_{ij}$ is also compared with elements $v_{hj}$ for $h>i$, when $v_h$ is to be inserted into $T^{(h-1)}$. 
Consider the first group. We compare $v_{ij}$ with $v_{hj}$ when 
$v_{h'j} \neq v_{hj}$ for all $h'$ less than $h$. 
In the second group, $v_{ij}$ is compared with $v_{hj}$ 
when $h>i$ and if $v_{hj} \neq v_{h'j}$ for all $h'<i$. 

Hence, the two algorithms perform the same comparisons.

\end{proof}

We can now prove Theorem \ref{thm:cmpalg}.

\begin{proof}[\textbf{Proof of Theorem} \ref{thm:cmpalg}]
By Proposition \ref{prop:sigmaalg}, the $\Sigma$-algorithm uses at most 
$nm + m^2$ operations. By \cite{Felszeghy}, the Point trie algorithm uses at most $nmr$ comparisons. The result now follows from Proposition \ref{prop:equalcmp}. 
\end{proof}


Since it is not possible to know the value $r$ a priori, it is not clear whether we should use the method from 
Proposition \ref{prop:sortalg} or the $\Sigma$-algorithm 
in the case when $\Omega$ is ordered and $n\log(m) < m^2$. 
It actually depends on the configuration of the points.

\end{document}